\documentclass[11pt]{amsart}
\usepackage[margin=3cm]{geometry}
\title[The $k$-Plancherel measure and a Finite Markov Chain]{The $k$-Plancherel Measure and a Finite Markov Chain}
\author{Svante Linusson, Alperen \"{O}zdemir} 
\address{Department of Mathematics, KTH Royal Institute of Technology, Stockholm, Sweden} 
\email{linusson@kth.se, alpereno@kth.se} 
\thanks{We would like to thank Jennifer Morse for helpful discussions, in particular help with the proof of Lemma \ref{conj}. Both authors were funded by the Swedish Research Council, VR, grant 2022-03875, and the second author is also funded by the Knut and Alice Wallenberg Foundation.}
\usepackage{amssymb, amsmath, enumerate, float, graphicx, qtree, mathtools, pgfplots, tikz, ytableau, xcolor, genyoungtabtikz, mathdots}
\usepackage[colorlinks=true, urlcolor=black,linkcolor=blue, citecolor=blue]{hyperref}

\date{2025-12-30}

\newtheorem{theorem}{Theorem}[section]
\newtheorem{lemma}[theorem]{Lemma}

\newtheorem{corollary}[theorem]{Corollary}
\newtheorem{definition}[theorem]{Definition}
\newtheorem{example}[theorem]{Example}

\newtheorem{conjecture}{Conjecture}[]

\newcommand{\vac} {\ensuremath{\cdot}}
\newcommand{\occ} {\ensuremath{\bigcirc}}

\newcommand{\bbox}{\hfill $\Box$}
\newcommand{\pf}{\noindent {\it Proof:} }

\newcommand{\cal}[1]{\mathcal{#1}}
\newcommand{\tn}[1]{\textnormal{#1}}

 \pgfplotsset{compat=1.18}

\begin{document}

\begin{abstract}
Let $\mathcal{P}_k(n)$ denote the set of partitions of $n$ whose largest part is bounded by $k,$ which are in 
well-known bijection with $(k+1)$-cores $\mathcal{C}_k$. We study a growth process on $\mathcal{C}_k$, whose stationary distribution is the $k$-Plancherel measure, which is a natural extension of the Plancherel measure in the context of $k$-Schur functions. When $k\to\infty$ it converges to the Plancherel measure for partitions, a limit studied first by Vershik-Kerov. 
However, when $k$ is fixed and $n\to \infty$, we conjecture that it converges to a shape close to the limit shape from the uniform growth of partitions, as studied by Rost. 
We show that the limiting behavior, for fixed $k$, is governed by a finite Markov chain with $k!$ states over a subset of the $k$-bounded partitions or equivalently as a TASEP over cyclic permutations of length $k+1$. 
This paper initiates the study of these processes, state some theorems and several intriguing conjectures found by computations of the finite Markov chain.

\end{abstract}

\maketitle

\section{Introduction}

We study an infinite growth process on $(k+1)$-cores, which are integer partitions with no hook of length $k+1$. They are in a well known bijection with $k$-bounded partitions. The two most classical growth models on partitions have limit shapes determined in famous research by Vershik-Kerov \cite{VK77} and Rost \cite{R81}. In the former the partitions were generated with distribution given by the Plancherel measure and in the latter the growth process was locally uniform in each step. They give rise to two different limit shapes of the diagram of the partition. Since then, these models have been the focus of a large number of fascinating and influential research, such as Johansson \cite{J00}. 

In \cite{AL14,L15} the modulo $k$ version of the Rost model was studied in connection to a problem on reduced random walks in the hyperplane arrangement of affine Coxeter groups. The limit shape $D_k$ for a fixed $k$ was shown to be a piecewise linear curve with $k-1$ straight lines of deterministic slope. The  shape $D_k$ closely followed the limit shape of Rost, and in fact, as $k\to \infty$ we come arbitrarily close to that limit shape.

The study of a similar growth modulo $k+1$ process defined so that the stationary distribution of the $(k+1)$-cores in limit will be a version of the Plancherel measure, which we call the $k$-Plancherel measure. Its stationary distribution converges to the Plancherel measure as $k\rightarrow \infty.$ However, as $n\rightarrow \infty,$ we show that it is governed by a finite Markov chain in the symmetric group $S_k.$ Our study is heavily based on the theory of $k$-Schur functions and weak and strong orders of the affine symmetric group; see survey \cite{LLMSSZ14}, as explained in Section \ref{sec:k-Planch}.  It appears that we get the same limit shape $D_k$ for a fixed $k$, Conjecture \ref{conj:limitshape}. That is, close to the shape of Rost instead of Vershik-Kerov.

Thanks to the so called $k$-rectangle property we can restrain the infinite process to a finite process on $k!$ states. They can be described most easily as $k$-bounded partitions that have at most $k-i$ parts of size $i$ for each $1\le i\le k$.  Just as in \cite{AL14,L15} we have a natural map from the process on the $(k+1)$-cores to a cyclic TASEP (Totally asymmetric exclusion process) on a ring with a permutation of the numbers $1,\dots,k+1$. This leads to a very interesting TASEP, with transition rates given by ratios of the number of strong marked tableaux of the shapes by the corresponding $(k+1)$-cores.  As explained in Section \ref{sec:TASEP}. See Figure \ref{fig:k=3} for the finite chain when $k=3$.

In this exposition, we have more conjectures than theorems about these processes. The stationary distribution of this finite process (TASEP or on $k$-bounded partitions) has several interesting properties, some conjectural. We may use what is called the $k$-conjugation in the theory of $k$-Schur functions. Lemma \ref{conj} tells that both the infinite and finite chains are symmetric with respect to the $k$-conjugation, which implies, for example, that the limit shape of the infinite process is symmetric.
For the finite $k$-chain we conjecture that there is also another involution which we call $k$-complementation (reverses the permutation in TASEP). It has no fix points and matches states that seem to have the same stationary probability, but for no apparent reason, see Conjecture \ref{conj:complement}. 
Another intriguing pattern is that the stationary distribution for the finite $k$-chain has all denominators which are multiples of small primes, and can all be written as, $\frac{A_\lambda}{\prod_{j=1}^k\binom{2j}{j}}$ for some integer $A_\lambda$, see Conjectures \ref{conj:lcd} and \ref{conj:minimum}. We have yet no conjecture of what $A_\lambda$ counts that can imitate the beautiful theory of multi-line queues by Ferrari and Martin \cite{FM07} for cyclic TASEP with uniform jump rates.

We list the theorems and conjectures we have in Sections \ref{sec:TASEP} and \ref{sect:lim}. The proofs are mostly based on so called raising operators from the theory of $k$-Schur functions and we prove some technical properties we need in an Appendix.

\section{Background}

\subsection{$k$-Schur functions}
The $k$-Schur functions are introduced in the context of proving positivity of Macdonald symmetric polynomials in \cite{LLM03}, 
and later found many other applications in geometry and representation theory. We will explore them from a probabilistic perspective.
See \cite{LLMSSZ14} for a comprehensive and very detailed survey of $k$-Schur functions. Besides all the interesting connections, regarding our particular interest, we will broach the topic through the usual Schur functions. 
We start with defining the homogenous symmetric functions for a set of possibly infinite variables 
$X=\{x_1,x_2, \ldots\}$ as
\[h_m[X]=\sum_{1 \leq i_1 \leq \ldots \leq i_m} x_{i_1}\ldots x_{i_m}.\]
Through which first we can define the ring of symmetric functions $\mathbb{Q}[h_1,h_2,\ldots]$.
We can extend the definition of $h_m$ to any partition $\lambda=(\lambda_1,\ldots \lambda_{l(\lambda)})$ of $\lambda_1+\ldots \lambda_{l(\lambda)}$ by letting  
$h_{\lambda}[X]=h_{\lambda_1}[X]\cdots h_{\lambda_{l(\lambda)}}[X].$ 
 Then one among many other means to define the Schur functions, is to let 
 \[s_{\lambda}[X]=\tn{det}\left(h_{\lambda_{i}+i-j} \right)_{1\leq i,j \leq n}\]
where $n \geq l(\lambda).$ We will express the determinant through the raising operators, which are defined as
\begin{equation} \label{raising}
R_{ij}(\lambda)=(\lambda_1,\ldots,\lambda_i+1,\ldots,\lambda_j-1,\ldots,\lambda_{l(\lambda)})    
\end{equation}
and we let
\begin{equation} \label{opext}
    R_{ij}h_{\lambda}=h_{R_{ij}(\lambda)}.
\end{equation}
So an alternative expression will be 
\[s_{\lambda}[X]=\prod_{i <j} (1-R_{ij}) h_{\lambda}.\]
 See Section I.4. of \cite{M98}. Note that the Schur functions form a basis for $\mathbb{Q}[h_1,h_2,\ldots].$ 
 
 The $k$-Schur functions are defined over \textit{$k$-bounded partitions} instead. A partition $\lambda$ is called $k$-bounded if $\lambda_1 \leq k.$ We let 
 \begin{equation} \label{deter}
     s_{\lambda}^{(k)}= \prod_{i=1}^{l(\lambda)} \prod_{j=i+1}^{k-\lambda_i+i} (1-R_{ij}) h_{\lambda}.
 \end{equation}
 They form a basis of 
\[\mathbb{Q}[h_1,\ldots,h_k],\]
which is a subring of the ring of symmetric functions. See Section $3$ of the aforementioned monograph \cite{LLMSSZ14} for a more direct, parametric definition of $k$-Schur functions. For our purposes, we will also give its definition by tableaux as well.

\subsection{Core partitions} \label{sec:core}

 We say $\kappa$ is an \textit{$r$-core partition} if there is no cell with hook-length equal to $r$ in its Young diagram. Let us denote the set of $r$-core partitions by $\mathcal{C}_{r},$ and the set of all $k$-bounded partitions by $\mathcal{P}_k$. In addition, we use  $\mathcal{P}_k(n)$ for the set of all $k$-bounded partitions of $n.$ There is a bijection between $\mathcal{C}_{k+1}$ and $\mathcal{P}_k,$ as explained in \cite{LM05}. For $\kappa \in \mathcal{C}_{k+1},$ it can be defined as
\[ \mathfrak{p}: \, \kappa \rightarrow (\lambda_1,\ldots,\lambda_l) \in \mathcal{P}_k\]
where $\lambda_i$ is the number of cells with hook-length smaller than 
to $k+1$ in the $i$th row 
of $\kappa.$ This corresponds to removing all boxes with hook-length greater than $k+1,$ and then sliding all cells to the left of the diagram. 

The inverse of $\mathfrak{p}$, let us call it $\mathfrak{c}$, can be described by a simple sequential process. We start from the top row, $\lambda_{l(\lambda)}$, and move down to the bottom row sequentially. We mark the hook-lengths in the row in the current stage, say $\lambda_i \tn{ for } i\leq l$. If we have a box with hook-length greater than or equal to $k+1$ in $\lambda_i$ we add boxes to the left of the diagram to all rows $\lambda_1,\ldots,\lambda_i$ until all cells already-present in $\lambda_i$ have hook-lengths less than or equal to $k.$ 

\begin{figure}
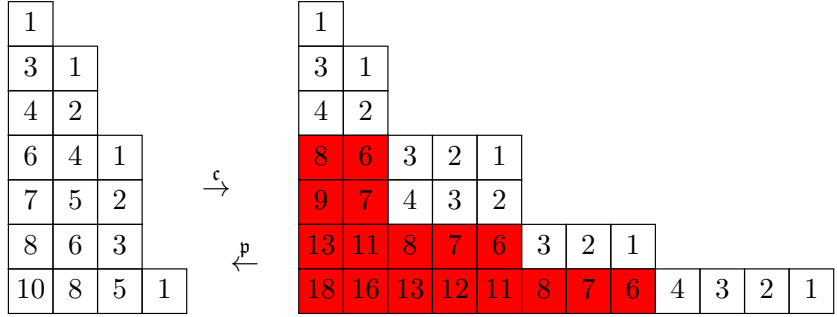
%[H]
\centering
\vspace*{0.5cm}

\begin{ytableau}
1 \\ 3 & 1 \\ 4 & 2 \\  6 & 4  & 1 \\  7 & 5 & 2\\  8 & 6 & 3\\  10 & 8 & 5 & 1  
\end{ytableau} 
$ \, \raisebox{-2cm}{$\overset{\mathfrak{c}}{\to}$}
\raisebox{-3cm}{$\overset{\mathfrak{p}}{\leftarrow}$} \quad $
\begin{ytableau}
1 \\ 3 & 1 \\ 4 & 2 \\ *(red) 8 & *(red) 6 & 3 & 2  & 1 \\ *(red)  9 & *(red) 7 & 4 & 3 & 2\\ *(red)  {13} & *(red)  {11} &*(red)  8 & *(red) 7 &  *(red) 6 & 3 & 2 & 1\\ *(red) {18} & *(red) {16} & *(red) {13} & *(red) {12} & *(red) {11} & *(red) 8 & *(red) 7 & *(red) 6 & 4 & 3 & 2 & 1  
\end{ytableau}
\caption{The Young diagram of a $4$-bounded partition and the map to a $5$-core partition $\kappa$ with hook lengths. The red boxes have hook length greater than 5 and are removed by $\mathfrak{p}$ and added by $\mathfrak{c}$.}
\end{figure}

\subsection{Weak and strong orders}

We are interested in two different orders that come essentially from the affine symmetric group $\tilde{S}_n,$ the weak order and the strong (Bruhat) order. We refer the reader to \cite{W99} for their algebraic and lattice theoretical aspects in $\tilde{S}_n$, and give the definitions in our context, which are based on \cite{L10} and \cite{LLMSSZ14}.

We use the French notation for Young diagrams as in the figures above, which is more convenient for bounded partitions. Recall that a \textit{content} of a cell $c=(i,j)$ in the $i$th row and the $j$th column of a Young diagram is $j-i.$ For the core partitions, we define

\begin{definition} \tn{(Weak order)} \label{weak order}
Let $\tau$ and $\kappa$ be $(k+1)$-cores. We say $\kappa$ precedes $\tau$ in the weak order, denoted by $\kappa \rightarrow_k \tau$ if and only if all cells in $\tau \setminus \kappa$ have the same content modulo $(k+1)$. 
\end{definition}

To give a similar definition for bounded partitions and later to compare it with the strong order, we first define a conjugation operation on them, which can be found in \cite{LM05}. We define it through $k$-cores, which is symmetric under the operation, unlike in the case with bounded partitions. Recall that for a given partition $\lambda,$  the \textit{conjugate} of $\lambda,$ denoted by $\lambda',$ is the partition whose tableau is the diagonal reflection of the tableau of $\lambda.$ 
\begin{definition} \label{conjugate}
Let $\lambda$ be a $k$-bounded partition. The \textit{$k$-conjugate} of $\lambda$ is defined as $\lambda^{\omega_k}=\mathfrak{p}(\mathfrak{c}(\lambda)').$ 
\end{definition}
Then we have
\begin{definition} \tn{(Weak cover for bounded partitions)}
Let $\lambda$ and $\mu$ be $k$-bounded partitions. We say $\mu$ covers $\lambda$ in weak order, denoted by $\lambda \rightarrow_k \mu$ if and only 
\[|\mu|=|\lambda|+1, \lambda \subseteq \mu \tn{ and }\lambda^{\omega_k} \subseteq \mu^{\omega_k}.\]
\end{definition}
Observe that in the usual Young lattice, we already have
\[\lambda \subseteq \mu \tn{ if and only if } \lambda' \subseteq \mu'.\]

Then we define the strong order on cores through bounded partitions.
\begin{definition} \tn{(Strong cover for cores)}
Let $\kappa$ and $\tau$ be $(k+1)$-cores. We say $\kappa$ covers $\tau$ in strong order, denoted by $\tau \Rightarrow_k \kappa$ if and only if 
\[|\mathfrak{p}(\kappa)|=|\mathfrak{p}(\tau)|+1 \tn{ and } \tau \subseteq \kappa.\]
\end{definition}

We note that the lattice formed by the weak order is a subposet of the strong order, that is, every cover relation in the weak order is also present in the strong order.
In the case of weak order, we have an edge between two partitions only if one is obtained from the other by box addition, and there are at most $k$ covers for each $(k+1)$-core on the lattice. See \cite{LLMSSZ14} for a comparison of two posets with the Young lattice.
 
\subsection{k-tableaux}

The Schur functions can be defined more combinatorially through semi-standard Young tableaux as 
 \begin{equation}\label{ssyt}
     s_{\lambda}[X]=\sum_{T \tn{ is }SSYT(\lambda)} x^T.
 \end{equation}
See Section 7.10 of \cite{S99} for more details. For our particular interest, we note that if we take the coefficient $x_1 x_2 \cdots x_n$ of $s_{\lambda}[X]$ where $\lambda \vdash n,$ we have instead a count of the standard tableaux, which is known as the \textit{dimension} of $\lambda$: 
 \begin{equation}\label{syt}
    d_{\lambda} = [x_1 x_2 \ldots x_n]s_{\lambda}[X]=\sum_{T \tn{ is }SYT(\lambda)} x^T.
 \end{equation}
It can also be calculated from the diagram using the hook-length formula as follows:
\[d_{\lambda}=\frac{n!}{\prod_{(i,j) \in \lambda}h(i,j)}\]
where the product runs over all cells of the diagram of $\lambda$ and $h(i,j)$ is the hook-length of the $(i,j)$th cell.

 Now we turn to $k$-analogs of semi-standard Young tableaux. The weak \tn{$k$-tableaux}  for core partitions is introduced in \cite{LM05}, then accompanied by strong tableaux in \cite{L10}.We will cover only the latter for our purposes and refer to Section 9 of \cite{L10} for a complete treatment of the topic with examples.

We recall the definition of a skew diagram. Given that $\tau \subset \kappa,$ let $\kappa \setminus \tau$ denote the skew diagram obtained by removing the cells of $\tau$ from $\kappa.$

\begin{definition}
We call a triplet $(\kappa, \tau, c)$ a strong marked cover and $c$ a marking if and only if $\kappa \Rightarrow_k \tau$ and $c$ is the content of the cell located at the southeast most of a connected component of $\kappa \setminus \tau.$ 
\end{definition}

\begin{definition} \label{def:tab}
A strong marked tableau of shape $\lambda \vdash n$ and weight $\beta=(\beta_1,\beta_2,\ldots,\beta_d)$ such that $\sum_{i=1}^d\beta_i=n$ is a sequence of $(k+1)$ cores
\begin{equation} \label{strchain}
\emptyset = \kappa^{(0)} \Rightarrow_k \kappa^{(1)} \Rightarrow_k \kappa^{(2)} \Rightarrow_k \cdots \Rightarrow_k \kappa^{(n)}=\mathfrak{c}(\lambda), 
\end{equation}
and an associated sequence of markings $\{c_1,\ldots,c_n\}$ such that
\begin{equation} \label{cond:cont}
   c_{A_i+1}  < c_{A_i+2} < \cdots <c_{A_i+ \beta_{i+1}}  
\end{equation}
where $A_0=0$ and $A_i=\sum_{j=1}^i \beta_i$ for $i=0,1,\ldots,d-1.$ 
\end{definition}
 Now we can express $k$-Schur functions analogous to \eqref{ssyt} as 
\[s_{\lambda}^{(k)}[X]=\sum_{T \tn{ is Strong}(\mathfrak{c}(\lambda))} x^T\]
where $\mathfrak{c}:\mathcal{P}_k \rightarrow \mathcal{C}_{k+1}$ is the bijection defined in Section \ref{sec:core} and $\textnormal{Strong}(\kappa)$ is the strong marked tableaux for the core partition $\kappa$ for a given set of variables. Let $\mathbf{K}_{\lambda \mu}^{(k)}$ be the number of strong marked tableaux of shape $\lambda$ and weight $\mu,$ which is an analog of Kostka numbers, the coefficients in the monomial expansion of Schur functions. For $\lambda \in \mathcal{P}_k(n),$ we have
\begin{equation} \label{strongkostka}
s_{\lambda}^{(k)}=\sum_{\mu }\mathbf{K}_{\lambda \mu}^{(k)} m_{\mu}
\end{equation}
similar to Schur functions. See Section 2.4 of \cite{LLMSSZ14}. We can also define a $k$-analog of the dimension in \eqref{syt} as
\begin{equation}\label{strongdim}
  d_{\lambda}^{(k)}:= [x_1 x_2 \cdots x_n] s_{\lambda}^{(k)}(x). 
\end{equation}
 Similarly, we define
 \begin{equation}\label{weakdim}
     F_{\lambda}^{(k)}[X]:=\sum_{T \tn{ is Weak}(\mathfrak{c}(\lambda))} x^T \quad \tn{ and } \quad w_{\lambda}^{(k)}:= [x_1 x_2 \cdots x_n] F_{\lambda}^{(k)}(x)
 \end{equation}
where $\textnormal{Weak}(\kappa)$ is the weak tableau for the core partition $\kappa$ for a given set of variables. We have the expansion:
\[F_{\lambda}^{(k)}=\sum_{\mu \in \cal{P}_k} K_{\lambda \mu}^{(k)} m_{\mu} \]
where the coefficients $K_{\lambda \mu}^{(k)}$ are known as \textit{weak Kostka numbers}. $K_{\lambda \mu}^{(k)}$ counts the number of sequences of bounded partitions, starting from $\emptyset$ and ending at $\lambda,$ with conditions imposed by $\mu$ on skew diagrams between successive elements of the sequence. We refer to Section 2.1-3 of \cite{LLMSSZ14} for details. If we particularly look at $w_{\lambda}^{(k)}=K_{\lambda,(1^n)}$, we observe that it counts the number of sequences precisely in the form
\[\emptyset = \lambda_1 \rightarrow \lambda_2 \rightarrow \cdots \rightarrow \lambda_{n-1} \rightarrow \lambda_n =\lambda,\]
which is the number of paths leading to $\lambda$ in the lattice of the weak order on $k$-bounded partitions. Either from this fact or, using the well-known bijection between affine Grassmannian elements and bounded partitions (See Sect. 1.2 of \cite{LLMSSZ14}), from the recursion for the number of reduced words for affine Grassmannian elements founded in Sect. 3.1 of \cite{BB05}, we have
\begin{equation}\label{w_kostka}
    w_{\lambda}^{(k)}= \sum_{\mu \to_k \lambda} w_{\mu}^{(k)}. 
\end{equation}

Finally, we note the following inequality 
\[w_{\lambda}^{(k)} \leq d_{\lambda} \leq d_{\lambda}^{(k)}\]
for all $k$-bounded partition $\lambda,$ which follows, for instance, from Theorem 5.5 of \cite{BB96}. We have 
$w_{\lambda}^{(k)} = d_{\lambda}^{(k)}=d_{\lambda}$
provided that $k\geq n.$

\section{A growth process on bounded partitions and a finite Markov chain}

\subsection{Cauchy identity and the $k$-Plancherel measure}\label{sec:k-Planch}

Now we will define a probability measure over bounded partitions analogous to Plancherel measure over partitions, equivalently over core partitions by the bijection above, which will involve both weak and strong orders. First, we invoke the well-known Cauchy identity for Schur functions:

\begin{equation}\label{cauchy}
    \prod_{i,j}\frac{1}{1-x_iy_j}=\sum_{\lambda}s_{\lambda}(x)s_{\lambda}(y)
\end{equation}
where $x=(x_1,x_2,\ldots)$ and $y=(y_1,y_2,\ldots)$ are possibly infinite sequences of variables. Following the Robinson-Schensted-Knuth algorithm or the fact \eqref{ssyt} noted above, we can give a tableau formulation to the identity as follows:
\[\prod_{i,j}\frac{1}{1-x_iy_j}=\sum_{(P,Q) \textnormal{ are SSYT}}x^P y^ Q = \sum_{\lambda} \left(\sum_{\textnormal{sh }P=\lambda} x^P\right) \left(\sum_{\textnormal{sh }Q=\lambda} y^Q\right)\]
where the sum runs over pairs of semistandard Young tableaux (P,Q) and $\textnormal{sh} \, P$ stands for the shape of the tableau $P.$ See Section 4.8 of \cite{S01} or Section 7.11 of \cite{S99} for a complete exposition of the subject.

What we are mainly interested in are the quantities obtained by taking the coefficient of $x_1 x_2 \cdots x_n y_1 y_2 \cdots y_n$ in \eqref{cauchy}. Using \eqref{syt} and noting that
\[[x_1 x_2 \cdots x_n y_1 y_2 \cdots y_n] \prod_{i,j}\frac{1}{1-x_iy_j}=n!,\]
we have 
\[\sum_{\lambda \vdash n}d_{\lambda}^2=n!.\]
Taking $d_{\lambda}$ as the dimension of the irreducible representation associated with $\lambda,$ the identity above restates the general fact that ``the sum of squared dimensions of irreducible representations of a group is equal to its size".

This immediately gives a measure on partitions of $n,$ by taking 
\begin{equation}\label{plancherel}
 \mathbf{P}(\lambda)=\frac{d_{\lambda}^2}{n!},   
\end{equation}
which is the celebrated Plancherel measure. This measure can be obtained sequentially as a growth process over Young diagrams. Suppose $\Lambda \vdash n+1$ can be obtained from $\lambda \vdash n$ by adding a box to its diagram. Then letting
\[P(\lambda,\Lambda)=\frac{d_{\Lambda}}{(n+1) \cdot d_{\lambda}},\]
we obtain a well-defined process on the set of all Young diagrams and its stationary distribution is the Plancherel measure. This can be shown from the relations
\begin{equation}\label{growth}
d_{\Lambda} = \sum_{\lambda: \lambda \nearrow \Lambda} d_{\lambda} \quad \textnormal{ and } \quad d_{\lambda} = \frac{1}{|\Lambda|} \sum_{\Lambda: \lambda \nearrow \Lambda} d_{\Lambda}.
\end{equation}
See, for instance, \cite{K94} and \cite{K96}.

Along the same lines, we can define an analogous measure over the bounded partitions. In Chapter 4 of \cite{L10}, the authors generalized RSK algorithm to bounded partitions and derive the analogous Cauchy identity, from which we get
\[n![x_1 x_2 \cdots x_n y_1 y_2 \cdots y_n]\sum_{\lambda \in \mathcal{P}_k}s_{\lambda}^{(k)}(x)s_{\lambda}^{(k)}(y)= \sum_{\lambda \in \mathcal{P}_k(n)} w_{\lambda}^{(k)} d_{\lambda}^{(k)},\]
see \eqref{strongdim} and \eqref{weakdim}. Then we define the \textit{$k$-Plancherel measure} for any $\lambda \in \mathcal{P}_k(n)$ as
\begin{equation}\label{k-plancherel}
 \mathbf{P}_k(\lambda)=\frac{w_{\lambda}^{(k)}d_{\lambda}^{(k)}}{n!}.  
\end{equation}

\subsection{Pieri rule and the growth process} \label{sec:pieri}

An important property which $k$-Schur functions inherit from the Schur functions is the Pieri rule:
\[h_1 s_{\lambda}^{(k)} = \sum_{\lambda \to_k \Lambda} s_{\Lambda}^{(k)}.\]
See Section 2.2 of \cite{LLMSSZ14}. Note that $\Lambda$ is obtained from $\lambda$ in weak order. Suppose $|\lambda|=n.$ We have
\begin{align*}
[x_1\ldots x_{n+1}] \, h_1 s_{\lambda}^{(k)}(x_1,\ldots,x_{n+1}) &= \sum_{i=1}^{n+1} [x_1 x_2 \cdots x_{i-1}x_{i+1}\cdots x_n x_{n+1}] s_{\lambda}^{(k)} = (n+1) d_{\lambda}^{(k)}.
\end{align*}
Therefore, we can define transition probabilities
\begin{equation}\label{inftrans}
    P(\lambda,\Lambda)= \frac{d_{\Lambda}^{(k)}}{(n+1)d_{\lambda}^{(k)}}.
\end{equation}
Its stationary distribution is given in \eqref{k-plancherel}, which can be easily verified using \eqref{w_kostka} as follows:
\[ \mathbf{P}_k(\lambda)=\sum_{\mu \to_k \lambda}\pi(\mu) P(\mu,\lambda)= 
\frac{w_{\mu}^{(k)}d_{\mu}^{(k)}}{(n-1)!} \frac{d_{\lambda}^{(k)}}{n \cdot d_{\mu}^{(k)}} = 
\frac{d_{\lambda}^{(k)}}{n!}  \sum_{\mu \to_k \lambda} w_{\mu}^{(k)} = \frac{w_{\lambda}^{(k)}d_{\lambda}^{(k)}}{n!}.
\]

\subsection{$k$-rectangle property and a finite Markov chain}\label{sec:rectangle}

Let $\square_i$ denote the partition $(i,i,\ldots,i)$ with $k-i+1$ parts for $i=1,\ldots,k.$ We have the following property of $k$-Schur functions: \\
\noindent \textbf{$k$-rectangle property} (Theorem 40 of \cite{LM07}): For any partition $\lambda \in \mathcal{P}_k$ and $i \in [k],$
\[s_{\square_i} s^{(k)}_{\lambda} = s^{(k)}_{\lambda \cup \square_i}. \]
See Corollary 8.3 of \cite{L08} for the equivalent property for affine Grassmanian permutations.

Recall that $\mathcal{P}_k$ denotes the set of $k$-bounded partitions. By the rectangle property, we can project the Markov chain to partitions such that the number of its parts with length $i$ is less than $k-i+1.$ Letting $\lambda=(\lambda_1,\lambda_2,\ldots)$ and $l_i(\lambda)$ denote the number of parts with length $i$, we define 
\[\mathcal{R}_k:=\{\lambda \in \mathcal{P}_k \, : \, l_i(\lambda)\leq k-i \tn{ for all } i=1,\ldots,k\}. \]
Since $|\mathcal{R}_k|=k!$, we obtain a Markov chain on $k!$ partitions, and the transition probability for any $\lambda,\mu \in \mathcal{R}_k$ is given by 
\[P(\lambda,\mu)= \frac{1}{ |\lambda|+1} \times \begin{cases}
  d_{\mu}^{(k)} \slash d_{\lambda}^{(k)} & \tn{ if  }\lambda \rightarrow_k \mu, \\
    d_{\Lambda}^{(k)} \slash d_{\lambda}^{(k)} & \tn{ if  }\lambda \rightarrow_k \Lambda = \mu \cup \square_i \tn{ for some }i \in [k],\\
    0  & \tn{ otherwise}.
\end{cases}\]
Let us call the stationary distribution of the Markov chain $\pi$. 

In general, it is difficult to find explicit expressions for $d_\lambda^{(k)}$ and for the transition rates. Nevertheless, one can use the raising operators, see Appendix \ref{appendix}, to identify them. Here we give one example and then one theorem.
\begin{example} \label{ex:dim}
 Let $k=3$ and $\lambda=(2,1,1)$. From \eqref{deter} and \eqref{strongdim} we know that $d_{(2,1,1)}^{(3)}$ is the 
 coefficient of $x_1x_2x_3x_4$ in $(1-R_{12})(1-R_{23})h_{(2,1,1)}=h_{(2,1,1)}-h_{(2,2,0)}-h_{(3,0,1)}+h_{(3,1,0)}$. First, note that the last two will cancel. We will use this sort of cancellation several times in later proofs. Since $h_{(2,1,1)}=h_2h_1h_1$, the coefficient of $x_1x_2x_3x_4$ is the number of ways to choose which variables should come from $h_2, h_1$ and $h_1$ respectively, that is $\binom{4}{2,1,1}$. Similarly, the coefficient in 
 $h_{(2,2)}$ is $\binom{4}{2,2}$ and hence $d_{(2,1,1)}^{(3)}=12-6=6$.
\end{example}

\begin{theorem} \label{Thm:rate1} Let $\lambda \in \mathcal{R}_k$ with $l_1(\lambda)=k-1$ and $\mu=\lambda \backslash (1^{k-1}).$ Then, we have $P(\lambda,\mu)=1/k.$ That is to say, adding a box in the first column of a partition with $(k-1)$ parts with size $1$, so that we remove $k$ 1's and $l_1$ becomes 0, is always $1/k$.
\end{theorem}
\begin{proof}
    We first note $\lambda_{l(\lambda)}= \lambda_{l(\lambda)-k+2}=1$ and $\lambda_{l(\lambda)-k+1} >1$ by the assumption. Let $\Lambda$ denote the $k$-bounded partition that is obtained by adding a box in the first column of $\lambda.$ Note that $\Lambda \notin \mathcal{R}_k.$ 

Suppose $|\lambda|=n$. Let us define 
\[\binom{n}{\lambda}=\frac{n!}{\lambda_1! \ldots \lambda_s!}.\]
Now putting \eqref{deter} and \eqref{strongdim} together we have
\[d_{\lambda}^{(k)}=[x_1\cdots x_n]\prod_{i=1}^{l(\lambda)} \prod_{j=i+1}^{k-\lambda_i+i} (1-R_{ij}) h_{\lambda}. \]
Then if we extend \eqref{opext} to the factorial term defined above, i.e.,
\[R_{ij} \cdot\binom{n}{\lambda}:=\frac{n!}{\lambda_1! \ldots (\lambda_{i}+1)! \ldots(\lambda_j-1)! \ldots  \lambda_s!},\]
we have
\[d_{\lambda}^{(k)}=\prod_{i=1}^{l(\lambda)} \prod_{j=i+1}^{k-\lambda_i+1} (1-R_{ij}) \cdot \binom{n}{\lambda}. \]
Let 
\[T_{\lambda}=\{(i,j) \in \mathbb{Z}^2 \, : \, 1 \leq i \leq l(\lambda), i+1 \leq j \leq k-\lambda_i+i\}\]
be the index of the operators in the product above. Then, for any subset $X \subseteq T_{\lambda},$ we define
\begin{equation} \label{r_x}
    \mathcal{R}_X=\prod_{(i,j)\in X} (-R_{i,j}).
\end{equation}
We also define the product of terms involving only $t$ parts of equal length with no smaller parts, that is, 
$\lambda_{l(\lambda)-t+1}=\dots=\lambda_{l(\lambda)}$. Let $s=l(\lambda)-t$.
\begin{equation}\label{r_triangle}
    \mathbf{R}^\triangle_t:=\prod_{i=1}^{t-1}\prod_{j=1}^{t-i} (1-R_{s+i,s+i+j}).
\end{equation}
If there exists $(i,j) \in X$ with $i\leq l(\lambda)-k+1$ and  $j \geq l(\lambda)-k+2,$ so if we move a box from some part of length 1 to a part with a length larger than 2, we have
\[\mathbf{R}_{k}^{\triangle }\circ R_X (\lambda) =0\]
by Lemma \ref{long_col1}. Similarly, we have $\mathbf{R}_{k+1}^{\triangle }\circ R_X (\Lambda) =0$ by Lemma \ref{long_col2}. Therefore, we only consider sums of products of the form $\mathbf{R}_{k+1}^{\triangle }\circ R_{X_1} (\lambda)+\mathbf{R}_{k+1}^{\triangle }\circ R_{X_2} (\lambda)+\cdots$ where 
\[X_1,X_2,\ldots \subseteq T_{\lambda} \cap \{(i,j) \in \mathbb{Z}^2 \, : \,  i \leq l(\lambda)-k+1, j \leq l(\lambda)-k+1\},\]
that is to say, with the cases in which no box from the parts of length 1 by the operators associated with parts of length larger than 1. 

Then we expand \eqref{r_triangle} using Lemma \ref{trieqk}. For a given interval $[i,j]$ in $S(U)$ the operators $R_{i,i+1}R_{i,i+2}\dots R_{i,j}$ means moving all the boxes on $i+1,\dots,j$ to $i$. Thus the effect of operators $R_{S(U)}$ on the column of 1's is collecting $j_s-i_s$ boxes on position $i_s$. Thus it gives a composition of $k-1$ into the length of the intervals of $S(U)$ and all numbers not 
belonging to an interval giving a term of size 1 in the composition.

Now, $l_1(\lambda)=k-1$, so write $\lambda$ as $\left(\mu,1^{k-1}\right)$ where $\mu$ is a partition with no parts of size 1. We have

\begin{equation*}
\begin{split}
d_{\lambda}^{(k)}=\sum_{U\subseteq [k-1]} (-1)^{|U|} R_{S(U)}(\lambda)&=\sum_{c_1+\dots+c_t=k-1} \binom{n-1}{c_1,c_2,\dots,c_t,\mu}\\
&=
\sum_{c_1+\dots+c_t=k-1} \binom{n-1}{k-1,\mu}\binom{k-1}{c_1,c_2,\dots,c_t}.
\end{split}
\end{equation*}

\[\]
By Lemma \ref{lem:comp} we get that
\[d_{\lambda}^{(k)}= \binom{n-1}{k-1,\mu}\cdot 1.\]
Similarly, we get $d_{\Lambda}^{(k)}= \binom{n}{k,\mu}$ and thus we get the transition rate
\[
P(\lambda,\Lambda)=\frac{\binom{n}{k,\mu}}{n\cdot \binom{n-1}{k-1,\mu}}=\frac{1}{k}.
\]
\end{proof}
\begin{lemma}\label{lem:comp} For a fixed $m\ge 1$ we have
 \[   \sum_{c_1+\dots+c_t=m} \frac{(-1)^{m-t}}{c_1!\dots c_t!}=\frac{1}{m!},
\]
where the sum is over all $2^{m-1}$ compositions of $m$ into positive integers.
\end{lemma}
\begin{proof}
    By induction over $m$. Obvious for $m=1$. For $m>1$ we sum over the value of $c_1$ and use induction.

 \begin{align*}
 \sum_{c_1+\dots+c_t=m} \frac{(-1)^{m-t}}{c_1!\dots c_t!}
 &=\frac{(-1)^{m-1}}{m!}+\sum_{c_1=1}^{m-1}\frac{(-1)^{c_1-1}}{c_1!}\sum_{c_2+\dots+c_t=m-c_1} \frac{(-1)^{m-c_1-t+1}}{c_2!\dots c_t!}=\\
 &=\sum_{c_1=1}^{m}\frac{(-1)^{c_1-1}}{c_1!}\frac{1}{(m-c_1)!}.
 \end{align*}
 The remaining sum is just the well known formula for summing binomial coefficients with alternating signs and the result follows.
\end{proof}

\subsection{TASEP formulation of the finite chain}\label{sec:TASEP}

We may also describe the $k!$ states of the Markov chain as cyclic permutations of length $k+1$. We will fix 
can rotate every cyclic permutation, so we may assume that the entry 
$(k+1)$ to always be the last in one-line notation and thus identify cyclic permutations of length $k+1$ with permutations of length $k$, using $S_k$ for both sets. 
We will define a TASEP (Totally Asymmetric Simple Exclusion Process) on cyclic permutations with the rule that a number $i$ can switch places with the number $j$ to its left only if $j>i$. That is, small numbers can move to the left.

We will now describe the bijection $\alpha :S_k\to \mathcal{R}_k$. First, we consider a modulo $k+1$ particle process on the integer line in which the particles may jump left one step. It starts with one particle on every non-negative integer position and no particles on the negative positions. There can never be two particles at the same position. The jump of a particle at position $r$ means that it moves to position $r-1$ if it is empty. When a particle at position $r$ jumps, all particles at positions equal to $r \pmod{k+1}$ will also jump if the position at $r-1\pmod{k+1}$ is empty. In general, for every state of this process, if a position $x$ has a particle, then all positions $y>x,y\equiv x \pmod{k+1}$ will have particles as well. 
This is the same process described in several papers, for example \cite{AL14,L15}, but we
need the particles to move left rather than right.

For any such state of particles on $\mathbb{Z},$ we bijectively associate a partition $\kappa$, in Russian notation. Its boundary on $\mathbb{R}^2$ is obtained from particles on $\mathbb{Z}$ by adding line segments in the southeast direction for vacant sites in $\mathbb{Z}$ and line segments in the northeast direction for occupied sites. It is well known that the partition created $\kappa$ is a $(k+1)$-core for any state of the process described in the previous paragraph. See Figure \ref{fig:modk} for an example. We will now give numbers $1,\dots,k+1$ to the particles iteratively for a given state. The particle furthest to the left will be assigned number 1 and so will all particles at the same position $\pmod{k+1}$. Then the unmarked particle furthest to the left will get number 2 and so on. This will give a mapping from particles with numbers $1,\dots,k+1$ to positions $r \pmod{k+1}$ for $r=0,\ldots,k$. Thus, a cyclic permutation $\pi$, a state in the TASEP. It is a well-known fact that the jumps in the $\pmod{k+1}$-process are in fact equivalent to moving from a core partition to another that covers it in the weak order for $(k+1)$-cores, see Def. \ref{weak order}.

From the $(k+1)$-core $\kappa$ we can apply the bijection $\mathfrak{p}(\kappa)$ to a $k$-bounded partition. Then, as a final step, we remove rectangles as described in 
Section \ref{sec:rectangle} to obtain a partition $\lambda\in \mathcal{R}_k$. Note that the a rectangle $(i,i,\dots,i)$ of $k+1-i$ parts in the $k$-bounded partition, corresponds to the particles numbered $1,\dots,k+1-i$ having passed numbers, in fact vacant positions, $k-+2-i,\dots,k+1$ one more time. The removal of the rectangle does hence not change the state of the TASEP. We conlude the following.

\begin{theorem}
    The map $\alpha :S_k\to \mathcal{R}_k$ that maps $\pi$ to $\lambda$ is a bijection.
\end{theorem}

In Figure \ref{fig:modk}, the permutation $1 4 2 3 5$ is mapped to the bounded partition $(4,3,1)$ and then removing 
a $4\times 1$-rectangle to $\lambda=(3,1)$.
Thus $\alpha(1 4 2 3 5)= (3,1)$.

A convenient way to find $\pi$ is to label the up-steps (northeast steps) in the $(k+1)$-core.
From that labeling, we also realize that when particle $i$ in the TASEP jumps left, it is the same as adding a box to all positions labeled $i$ in the $(k+1)$-core $\kappa$. That is the same as adding one box to the column $i$ of $\lambda$. Note how the permutation $\pi$ encodes which boxes can be added to the partition $\lambda$. In Figure \ref{fig:modk}, we see that we cannot add any box to column 3 because column 2 has the same height, which corresponds to that the larger number 3 cannot pass the smaller number 2 in the permutation. It is also not possible to add a box in the fourth column, which is more difficult to see directly from $\lambda$, but can be seen in the labeling of $\kappa$ and even more easily in $\pi$ where 4 has 1 on its left.

\begin{figure}

 \scalebox{0.8}{
\begin{tikzpicture} [scale=0.5]
\draw [thick] (-6.5,7.5) -- (0.5,0.5)--(8.5,8.5);
\draw (-6,0) node(-6) {$\vac$};
\draw (-5,0) node(-5) {$\vac$};
\draw (-4,0) node {$\vac$};
\draw (-3,0) node(-3) {$\vac$};
\draw (-2,0) node(-2) {$\occ$};
\draw (-1,0) node(-1) {$\vac$};
\draw (0,0) node(0) {$\occ$};
\draw (1,0) node(1) {$\occ$};
\draw (2,0) node {$\vac$};
\draw (3,0) node {$\occ$};
\draw (4,0) node {$\occ$};
\draw (5,0) node {$\occ$};
\draw (6,0) node {$\occ$};
\draw (7,0) node {$\vac$};
\draw (8,0) node {$\occ$};
\draw (9,0) node {$\occ$};
\draw (10,0) node {$\occ$};
\draw (11,0) node {$\occ$};
\draw (12,0) node {$\occ$};
\draw [thick] (-2.5,3.5)--(-1.5,4.5)--(1.5,1.5);
\draw [thick] (-1.5,2.5)--(1.5,5.5)--(3.5,3.5);
\draw [thick] (-0.5,1.5)--(6.5,8.5)--(7.5,7.5);
\draw [thick] (0.5,4.5)--(2.5,2.5);
\draw [thick] (3.5,5.5)--(4.5,4.5);
\draw [thick] (4.5,6.5)--(5.5,5.5);
\draw [thick] (5.5,7.5)--(6.5,6.5);
\draw (-6,-1.5) node {$3$};
\draw (-5,-1.5) node {$4$};
\draw (-4,-1.5) node {$0$};
\draw (-3,-1.5) node {$1$};
\draw (-2,-1.5) node {$2$};
\draw (-1,-1.5) node {$3$};
\draw (0,-1.5) node {$4$};
\draw (1,-1.5) node {$0$};
\draw (2,-1.5) node {$1$};
\draw (3,-1.5) node {$2$};
\draw (4,-1.5) node {$3$};
\draw (5,-1.5) node {$4$};
\draw (6,-1.5) node {$0$};
\draw (7,-1.5) node {$1$};
\draw (8,-1.5) node {$2$};
\draw (9,-1.5) node {$3$};
\draw (10,-1.5) node {$4$};
\draw (11,-1.5) node {$0$};
\draw (12,-1.5) node {$1$};
\draw (13.5,-1.8) node {$\ldots$};
%%%
\draw (-2,0) node {\scriptsize 1};
\draw (0,0) node {\scriptsize 2};
\draw (1,0) node {\scriptsize 3};
\draw (3,0) node {\scriptsize 1};
\draw (4,0) node {\scriptsize 4};
\draw (5,0) node {\scriptsize 2};
\draw (6,0) node {\scriptsize 3};
\draw (8,0) node {\scriptsize 1};
\draw (9,0) node {\scriptsize 4};
\draw (10,0) node {\scriptsize 2};
\draw (11,0) node {\scriptsize 3};
\draw (12,0) node {\scriptsize 5};
\draw (13.5,-0.3) node {$\ldots$};
\draw (-2.3,4.2) node {\scriptsize 1};
\draw (-0.3,4.2) node {\scriptsize 2};
\draw (0.7,5.2) node {\scriptsize 3};
\draw (2.7,5.2) node {\scriptsize 1};
\draw (3.7,6.2) node {\scriptsize 4};
\draw (4.7,7.2) node {\scriptsize 2};
\draw (5.7,8.2) node {\scriptsize 3};
\draw (11,3) node {$\overset{\mathfrak{p}}{\to}$};
\draw (15,3) node {$\begin{ytableau} \\  &  &  \\   & & & \\ \end{ytableau}$};
\draw (18,3) node {${\to}$};
\draw (22,3) node {$\begin{ytableau} \\  &  &  \\ \end{ytableau}$};
\end{tikzpicture}
}

    \caption{A state in the $\pmod5$-process with the corresponding $5$-core $(7,3,1)$ and cyclic permutation $1 4 2 3 5$. Followed by the 4-bounded partition image $(4,3,1)$ under $\mathfrak{p}$ and its projection to $(3,1)\in\mathcal{R}_k$ using the $k$-rectangle property. }
    \label{fig:modk}
\end{figure}
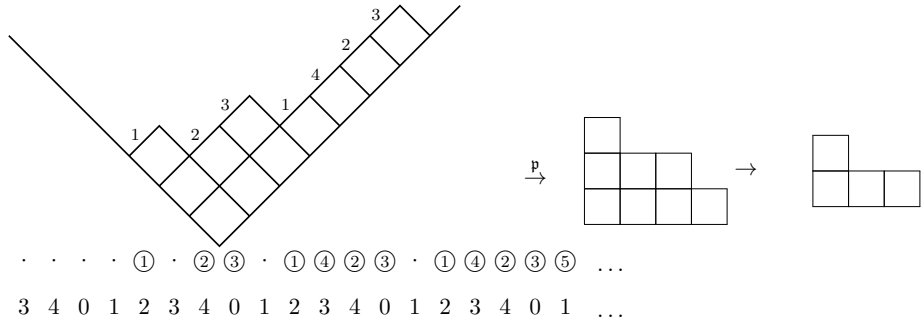

Let us now describe $\alpha^{-1}$ in a way that does not use $(k+1)$-cores.
First, recall that for a $k$-bounded partition $\lambda$ we use $l_i(\lambda)$ to denote the number of parts in $\lambda$ of size $i$. For every state $\lambda\in\mathcal{R}_k$ in the finite $k$-chain, we have $0\le l_i(\lambda)\le k-i$. For a state $\lambda$ we construct the cyclic permutation $\alpha^{-1}(\lambda)$ recursively as follows. At first, we only have $(k+1),$ which will remain in the last position through out. We first place $k$, for which there is no choice, so that the permutation becomes $k\ (k+1)$. Now, assume that we have already placed $i+1,\dots (k+1)$. Then we count from $i+1$ to the left $l_i(\lambda)$ (cyclic) steps and place $i$ there.
 \begin{example} Let $\lambda$ be the $5$-bounded partition with parts $3, 3, 1,1$ and thus the vector $(l_1(\lambda),\dots,l_4(\lambda))=(2,0,2,0)$. After the first step, we have $5\ 6$.
\\
Since $l_4(\lambda)=0$, 4 is placed directly to the left of 5, and we get   $4\ 5\ 6$.\\
Since $l_3(\lambda)=2$, 3 is placed two steps to the left of 4, counting cyclically we get   $4\ 3\ 5\ 6$.\\
Since $l_2(\lambda)=0$, 2 is placed directly to the left of 3, we get   $4\ 2\ 3\ 5\ 6$.\\
Finally $l_1(\lambda)=2$, so 1 is placed two steps to the left of 2, so we get   $\alpha^{-1}(3,3,1,1)=4\ 2\ 3\ 5\ 1\ 6$.
\end{example}

To see that this is indeed the inverse, note that if $i$ is $s$ steps to the left of $i+1$ corresponds to $i$ having jumped $s\pmod{k-i}$ times more than $i+1$. Note that we can ignore the numbers smaller than $i$ since none of them can pass them. This corresponds to that a box has been placed in column $i$ a total of $s\pmod{k-i+1}$ times (shift in modulus to include when $i$ jumps over $i+1$) and thus, after removing $k$-rectangles we get $l_i=s$.

As special cases, we find that the empty partition maps to $1\ 2\ \dots k\ (k+1)$ and the maximal partition maps to the reverse $k\ (k-1)\ \dots 2\ 1 \ (k+1)$.  
Note that it follows directly from the description of $\alpha^{-1}$ that  $\alpha^{-1}(\lambda^{\omega_k})$ is obtained by reversing the order of the elements of $\alpha^{-1}(\lambda)$, but keeping $(k+1)$ fixed. 

From the above discussion, one can show the following, which states that we are equivalently studying a TASEP on the ring $\mathbb{Z}_{k+1}$.

\begin{theorem}\label{thm:TASEP}
For $\lambda,\mu \in \mathcal{R}_k,$ $\mu$ can be obtained from $\lambda$ by adding a box to the column $i$ of $\lambda$ if and only if the jump of particles labeled $i$ to the left in the TASEP moves $\alpha^{-1}(\lambda)$ to $\alpha^{-1}(\mu)$. In particular, if $i$ has a smaller number on its left in $\alpha^{-1}(\lambda),$ then adding a box to column $i$ of $\lambda$ is not a valid transition in the chain. 
\end{theorem}
Therefore, the transitions in the TASEP are in agreement with the transitions in the Markov chain, so that if we assign the same probabilities to the TASEP with the $k$-chain, we have an equivalent process. 
\begin{corollary}
    A rectangle $i^{k+1-i}$ is created, and thus removed, in the finite $k$-chain exactly when particle $i$ passes particle $i+1$ in the TASEP-formulation.
\end{corollary}
\begin{proof}
    By Theorem \ref{thm:TASEP} we know that the height of the column $i$ minus the height of column $i+1$ is equal to the number of particles that $i$ has passed that $i+1$ has not passed. Hence, for a state $\lambda$ with $k-i$ parts $i$, the TASEP state $\alpha^{-1}(\lambda)$ has  
    $i$ to the right of $i+1$, with no particles labeled $i+2,\dots,k+1$ in between. A jump when $i$ passes $i+1$ gives one more box to the column $i$ and thus a rectangle $i^{k+1-i}$.
\end{proof}

\begin{example}
    Take $k=3$ and $\lambda=(2,1)$, see Figure \ref{fig:k=3}. There are at most $k=3$ possible transitions. The first possibility is to add a box in the first column (1 jumping in the TASEP) to get $(2,1,1)$. The transition rate for this is $\frac{d^{(3)}_{(2,1,1)}}{4\cdot d^{(3)}_{(2,1)}}=\frac{6}{4\cdot 2}=\frac{3}{4}$, see Example \ref{ex:dim}.

    The second possibility is adding a box in the second column (2 jumping) creating the partition $(2,2)$, which by the rectangle property is removed and we end up with the empty partition. The transition rate is similarly $\frac{2}{4\cdot 2}=\frac{1}{4}$.
    
    The third possibility would theoretically be adding a box in the third column (3 jumping) creating $(3,1)$ (and by removing a 3-rectangle, thus $(1)$). But this is not possible. 
    It can be seen either by moving to the corresponding $4$-core, 
    which is still $(2,1)$, and note that $(3,1)$ is not a 4-core and thus adding a box in the third column would not give us the 4-core associated with $(3,1)$.  Recall Definition \ref{weak order} and note that the only 4-cores possible from $(2,1)$ are $(3,1,1)$ and $(2,2)$, corresponding to the first two cases. Or we could see from the TASEP-formulation that 3 cannot jump because there is a smaller number to the left of it.
\end{example}

\begin{figure}
\begin{center}

\begin{tikzpicture}[scale=3, node distance=2cm]

% nodes
\node (A) at (0, 0) {\large 123\textcolor{gray}{4}};
\node (B) at (-0.866, 0.5) {\large 132\textcolor{gray}{4}};
\node (C) at (-0.866, 1.5) {\large 312\textcolor{gray}{4}};
\node (D) at (0, 2) {\large 321\textcolor{gray}{4}};
\node (E) at (0.866, 1.5) {\large 231\textcolor{gray}{4}};
\node (F) at (0.866, 0.5) {\large 213\textcolor{gray}{4}};

% nodes
\node (c) at (0, -0.15) { $\pi(\varnothing)=\frac{3}{20}$};
\node (c) at (-1.12, 0.35) { $\pi(2,1)=\frac{4}{20}$};
\node (c) at (1.12, 0.35) { $\pi(1,1)=\frac{3}{20}$};
\node (c) at (-1.12, 1.65) { $\pi(2)=\frac{3}{20}$};
\node (c) at (1.12, 1.65) {$\pi(1)=\frac{4}{20}$};
\node (c) at (0, 2.15) {$\pi(2,1,1)=\frac{3}{20}$};

% nodes
\node (c) at (0.39, 0.8) { $1$};
\node (c) at (-0.5, 0.2) { $\frac{1}{4}$};
\node (c) at (-0.38, 1.2) { $\frac{3}{4}$};
\node (c) at (-0.925, 1) { $\frac{2}{3}$};
\node (c) at (-0.39, 0.8) { $\frac{1}{3}$};
\node (c) at (-0.5, 1.8) { $\frac{1}{3}$};
\node (c) at (0.5, 1.8) { $\frac{1}{3}$};
\node (c) at (0.38, 1.2) { $\frac{1}{3}$};
\node (c) at (0.925, 1) { $\frac{1}{2}$};
\node (c) at (0, 1.4) { $\frac{1}{2}$};
\node (c) at (0, 0.6) { $\frac{2}{3}$};
\node (c) at (0.5, 0.2) { $\frac{1}{3}$};

% arrows
\draw[->]
  (A) edge (E) (B) edge (A) (B) edge (D) ;
\draw[->]
 (C) edge (A) (C) edge (B)  ;
\draw[->]
(D) edge (C) (D) edge (E)  (D) edge (F)  ;
\draw[->]
(E) edge (C) (E) edge (F)   ;
\draw[->]
(F) edge (B) (F) edge (A)   ;
\end{tikzpicture}

\vspace{-0.5cm}

\end{center}
\caption{The Markov chain on $3$-bounded partitions. The cyclic permutation $\alpha^{-1}(\lambda)$ is given for each partition and the stationary probability $\pi(\lambda)$.}
\vspace{-0.5cm}
\label{fig:k=3}
\end{figure}
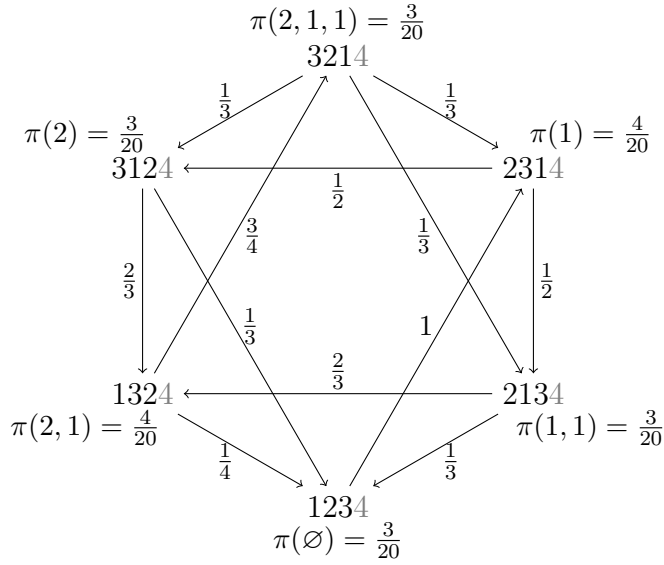

\begin{figure}[h!]
    \centering
    \includegraphics[width=0.9\textwidth]{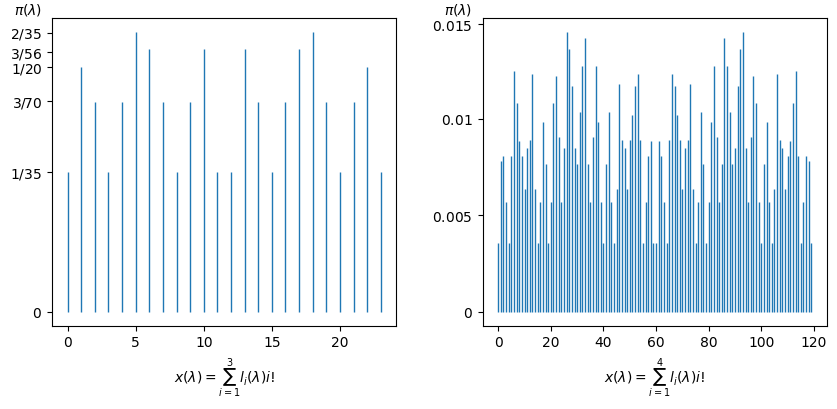}
    \caption{The stationary distribution of the Markov chain for $4$-bounded and $5$-bounded partitions from data. The partitions in $\mathcal{R}_k$ are indexed according to factorial base on the $x$-axis, see Section 4.1. of \cite{K14}, where $l_i(\lambda)$ is the number of parts of size i in $\lambda \in \mathcal{R}_k$.}
    \label{fourfive}
\end{figure}

\subsection{Conjugates and complements of partitions}
 Referring to Definition \ref{conjugate}, we note the following fact on $k$-conjugate partitions and the transition probabilities. We are surprised that we could not find this stated in the literature, and we thank Jennifer Morse for clarifying that this indeed follows from the characterization in terms of marked strong tableaux.
\begin{lemma} \label{conj}
For $\lambda,\mu \in \mathcal{R}_k,$ we have
$d_{\lambda}^{(k)}=d_{\lambda^{\omega_k}}^{(k)} \tn{ and } P(\lambda,\mu)=P(\lambda^{\omega_k},\mu^{\omega_k})$.
\end{lemma}

\pf By \eqref{strongdim}, we have
\begin{equation} \label{kost}
    d_{\lambda}^{(k)}=\mathbf{K}_{\lambda, (1,\cdots,1)}^{(k)} \cdot [x_1\cdots x_n]m_{(1,\cdots,1)}=\mathbf{K}_{\lambda, (1,\cdots,1)}^{(k)}.
\end{equation}
 So, if we take $\kappa=\mathfrak{c}(\lambda)$ as the associated core partition, we want to show that the number of strong marked tableaux of shape $\mathfrak{p}(\kappa)$ and $\mathfrak{p}(\kappa')$ with the same weight $\beta=(1,\ldots,1)$ are the same. Observe that, by the value of $\beta$ in the definition of strong marked tableaux, see Def. \ref{def:tab},  we can drop the condition \eqref{cond:cont} on the contents of markings. Then we observe that the conjugates of core partitions in \eqref{strchain} also form a chain of strong covers leading to $\kappa'$, since the conjugation does not change the containment condition in the definition of strong cover. 
 Again, because of $\beta$, the marking merely counts the number of connected components of skew diagrams $\kappa^{(i+1)} \backslash \kappa^{(i)}$ for $i=1,\ldots,n$, which is not changed either under conjugation. Then it follows from \eqref{kost} that $d_{\lambda}^{(k)}=d_{\lambda^{\omega_k}}^{(k)}$. For the second result, we first observe that  $\lambda \rightarrow_k \mu$ if and only if $\lambda^{\omega_k} \rightarrow_k \mu^{\omega_k}$ by the definition of weak order. Therefore, $ P(\lambda,\mu)=0$ if and only if $P(\lambda^{\omega_k},\mu^{\omega_k})=0.$ In case they are non-zero, we also have $ P(\lambda,\mu)=P(\lambda^{\omega_k},\mu^{\omega_k})$, which follows from the first result and \eqref{inftrans}. 
 
\bbox

 This lemma implies the following result.
\begin{theorem} \tn{(Conjugation symmetry)}
    For all $\lambda \in \mathcal{R}_k,$ $\pi(\lambda)=\pi(\lambda^{\omega_k}).$
\end{theorem}
\pf Let $P$ denote the transition matrix of the Markov chain over $\mathcal{R}_k$ for a fixed $k\geq 2.$ We define $P'$ to be the permutation of the entries of $M_k$ according to $k$-conjugation, that is to say,
\[P'(\lambda,\mu)=P(\lambda^{\omega_k}, \mu^{\omega_k}).\]
Since $k$-conjugation is an involutive operation, i.e., $\left(\lambda^{\omega_k}\right)^{\omega_k}=\lambda,$ we have $P'=P$ by Lemma \ref{conj}. Therefore, they have the same stationary distribution, which proves the theorem. 

\bbox

We now list a number of conjectures on $\pi$ for this finite $k$-chain, based on our computations for $k\le 6$. Let $M_k=\prod_{j=1}^k\binom{2j}{j}$.
For a $\lambda\in\mathcal{R}_k$, let $\lambda^{\textnormal{comp}}\in\mathcal{R}_k$ be such that 
$\lambda \cup \lambda^{\textnormal{comp}} = (k-1,(k-2)^2,\ldots,i^{k-i},\ldots,2^{k-2},1^{k-1})$, that is $l_i(\lambda)+l_i(\lambda^{\textnormal{comp}})=k-i$.
 \begin{conjecture} \label{conj:complement}\textnormal{(Symmetry of complements)}
   For all $\lambda\in\mathcal{R}_k$, $\pi(\lambda)=\pi(\lambda^{\textnormal{comp}}).$
\end{conjecture}

%\small % Gör allt mer kompakt

Table with stationary distribution (multiplied with common denominator 280), exhibiting the intriguing symmetry of complements in Conjecture \ref{conj:complement}.
\begin{center}
\scalebox{0.9}{
\begin{tabular}{@{}llr@{\hskip 1cm}lll@{}}
\hline
$\alpha^{-1}(\lambda)$ & $\lambda$ & $\pi(\lambda)\cdot 280$ & $\alpha^{-1}(\lambda^{\textnormal{comp}})$ & $\lambda^{\textnormal{comp}}$ \\
\hline
1234 & $\emptyset$           & 8  & 4321 & 3 2 2 1 1 1 \\
2341 & 1           & 14 & 1432 & 3 2 2 1 1   \\
2314 & 1 1         & 12 & 4132 & 3 2 2 1     \\
2134 & 1 1 1       & 8  & 4312 & 3 2 2       \\
3412 & 2           & 12 & 2143 & 3 2 1 1 1   \\
3142 & 2 1         & 16 & 2413 & 3 2 1 1     \\
1342 & 2 1 1       & 15 & 2431 & 3 2 1       \\
3421 & 2 1 1 1     & 12 & 1243 & 3 2         \\
3124 & 2 2         & 8  & 4213 & 3 1 1 1     \\
1324 & 2 2 1       & 12 & 4231 & 3 1 1       \\
3241 & 2 2 1 1     & 15 & 1423 & 3 1         \\
3214 & 2 2 1 1 1   & 8  & 4123 & 3           \\
\hline
\end{tabular}}
\end{center}

\begin{conjecture}\label{conj:lcd} For every $\lambda \in \mathcal{R}_k$, we have 
$\pi(\lambda)=A_\lambda/M_k$, for some integer $A_\lambda$.
\end{conjecture}

The transition rates in the finite $k$-chain can involve complicated rational numbers, but the stationary distributions all have denominators with small prime factors. The lowest common denominator for the stationary distributions are $20,280,70560, 310464$ for $k=3,4,5,6$ respectively. From these, and the minimal value below, we have conjectured $M_k$.
\begin{conjecture} \label{conj:minimum}
The minimal value of $\pi$ in the $k$-chain is $(k+1)\prod_{j=1}^k\binom{k}{j}/M_k$.
And moreover, the minimum value is obtained exactly for those partitions $\lambda$ where, for each $i$ we have either $l_i(\lambda)=0$ or $l_i(\lambda)=i-1$. In particular the minimum is obtained exactly $2^{k-1}$ times.
\end{conjecture}

There are also more involved patterns emerging in the data, such as the following.

\begin{conjecture}
The probability that $k$ is in position $j$ in the cyclic permutation $\alpha^{-1}(\lambda)$ is the same as the probability of  $l_1(\lambda)=j-1$, that is $\lambda$ has $j-1$ 1s.
\end{conjecture}

\section{Limit shape of core partitions}\label{sect:lim}
Let us consider the Markov chain on the infinite state space of $k$-bounded partitions, whose transition rates are given in \eqref{inftrans}. For a fixed $k$, this gives us a triplet of random variables at each stage, let us call it $(B_n,K_n,R_n),$ where $B_n$ is the random $k$-bounded partition at the $n$th stage, $B_n = \mathfrak{p}(K_n)$ and $R_n$ is the partition obtained from $B_n$ by removing its $i$-rectangles for $i\in [k].$ 
\begin{definition}
   Let $C_k(n)$ be the upper-right boundary of the Young diagram of $K_n$ that is placed in the positive quadrant of $\mathbf{R}^2$ with decreasing length of parts along the $y$-axis. We call  $\mathcal{C}_{k}=\lim_{n\rightarrow \infty} C_k(n)/n$  the \textnormal{limit shape} of the Markov chain if the limit exists.
\end{definition}

In order to make inferences on the limit shape, we turn to $R_n.$ For a given \[\lambda=(\ldots,(i-1)^{k_{i-1}}, i^{k_i},\ldots) \in \mathcal{R}_k,\]
assuming that adding a box to a part of length $i-1$ gives another partition in $\mathcal{R}_k$, we denote the partition obtained by
\[\lambda^{+i}=(\ldots,(i-1)^{k_{i-1}-1}, i^{k_i+1\tn{ (mod k-i+1})},\ldots) \in \mathcal{R}_k.\]
Then, for a fixed $i,$ we define the set of partitions with $k_i=k-i$, 
\[\mathcal{R}_{k}^i:=\{\lambda \in \mathcal{R}_k : |\lambda| > |\lambda^{+i}|  \}.\]
In other words, $\mathcal{R}_{k}^i$ is the set of partitions in the state space $\mathcal{R}_k$ of the finite $k$-chain such 
that the above operation creates a rectangle of size $k-i+1 \times i,$ and it is removed.
By the Ergodic Theorem, %see Theorem 6.4.(3) in \cite{GS20}, 
the mean occurrence of the event 
$\{R_n=\lambda \tn{ and }R_{n+1}=\lambda^{+i}\}$
is $\pi(\lambda) \cdot P(\lambda,\lambda^{+i}).$ Let us define
\[\rho_i:=\sum_{\lambda \in \mathcal{R}_{k}^i} \pi(\lambda) \cdot P(\lambda,\lambda^{+i}).\]
If $L_i(n)$ denotes the number of parts of $B_n$ that have length $i$
and $E_i$ the event that an $i$-rectangle is created, then  
$n^{-1} \cdot L_i(n)=n^{-1} \cdot \{\textnormal{the number of times }E_i \tn{ occured}\}\rightarrow \rho_i.$
By Lemma \ref{conj}, $\rho_i=\rho_{k-i}.$ This implies by the limit above the following.
\begin{theorem} 
The limit shape $C_k$ is symmetric around the diagonal.
\end{theorem}

Let $D_{k}$ be the piecewise linear curve with vertices 
$v_i=\gamma(\binom{i}{2},\binom{k-i+1}{2})$, for $i=1\dots k$, where $\gamma
$ is a scaling constant,
see Figure \ref{F:corelimit}.

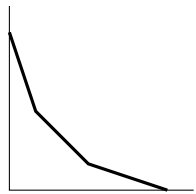
\begin{figure}[h]
\begin{center}
\begin{tikzpicture}[scale=0.7,rotate=90]
\draw[very thick] (0,0.5)--(0.5,2)--(1.5,3)--(3,3.5); 
\draw[thin](0,0)--(0,3.5)--(3.5,3.5);
\end{tikzpicture}
\end{center}
\caption{The limiting piecewise-linear curve $D_4$ for random 4-cores.}
\label{F:corelimit}
\end{figure}

In fact, data for $1\le i\le k\le 6$ show that a stronger statement might hold.
\begin{conjecture}\label{conj:rho}
    $\rho_i=\frac{1}{\binom{k+2}{3}}$ for all $i \in [k].$
\end{conjecture}
\begin{conjecture}\label{conj:limitshape}
    The limit shape $\mathcal C_k$ is $D_{k+1}$.
\end{conjecture}
 The first conjecture will imply the second, since it implies that in the limit there will be equally many $i$-rectangles for every $i$. By the symmetry of 
 Lemma \ref{conj} we know that $\rho_i=\rho_{k+1-i}$.
 Note that the sum of the area of the possible rectangles $k\times 1, k-1\times 2,\dots,1\times k$ is $\binom{k+2}{3}$ which agrees with the conjecture. Note also that by Theorem \ref{Thm:rate1} it would suffice to know the probability at stationarity that 1 is directly to the right of 2 in the TASEP, to prove Conjecture \ref{conj:rho} in the case of $i=1$ (and $i=k$ by symmetry). For other values of $i$ is seems more complicated, but never the less true.

\section{Appendix: Raising Operators} \label{appendix}
Here we list some lemmas on raising operators which were used for both results and computations. We defined the raising operator $R_{ij}$  in \eqref{raising}. Recall the product of operators $\mathbf{R}^\triangle_t$ defined in \eqref{r_triangle}. Expanding this product several terms will cancel. Recall that $R_{i,j}$ is moving a box from part $j$ to part $i$, which is also the effect of $R_{i,k}R_{k,j}$ for $i<k<j$. They will have opposite signs and therefore cancel each other.

We then recall that an inverse $i,j$ in a permutation $\pi$ is a pair $i<j$ such that $j$ comes before $i$ in one line notation of $\pi$ or equivalently $\pi^{-1}(j)<\pi^{-1}(i)$. Let $Inv(\pi)$ be the set of all inversions of $\pi$.

\begin{theorem}\label{thm:inversions} For every $t\ge 2$,
    \[\mathbf{R}^\triangle_t=\sum_{\pi\in S_{t}} (-1)^{|Inv(\pi)|}\prod_{(i,j)\in Inv(\pi)} R_{i,j}
    \]
\end{theorem}
\begin{proof}
First not that if we let $\binom{[t]}{2}$ denote all possible pairs $(i,j), 1\le i<j\le t$, then by definition 
$\mathbf{R}^\triangle_t=\sum_{S\subseteq \binom{[t]}{2}} (-1)^{|S|}\prod_{(i,j)\in S} R_{i,j}$.
We will now use matchings as discussed above to cancel terms so that only terms corresponding to sets of inversions remain.

We will define the matchings recursively and thereby prove the statement by induction over $t$.
For the base case $t=2$ we have that there are only to subsets of $\binom{[2]}{2}$ and both $\emptyset $ and $\{(1,2)\}$ are true sets of inversions of the two permutations of length 2.

Assume now that we have defined matchings for all subsets of $\binom{[t]}{2}$ leaving only the $t!$ sets corresponding to inversions of permutations in $S_t$. 

Some subsets of $\binom{[t+1]}{2}$ will be such that the subset of pairs not containing $t+1$ have been matched away in the previous step. For those sets, we can inductively use the same matchings. Any extra pairs containing   $t+1$ will not interfere with the matching.

All other subsets $S\subseteq \binom{[t+1]}{2}$ will consist of pairs $T\subseteq \binom{[t]}{2}$ and pairs $X$ containing $t+1$, where $T$ has not been matched away. This means that there exists $\pi\in S_t$ such that $Inv(\pi)=T$.  The set  $S=T\cup X$ will form the inversions of a permutation in $S_{t+1}$ if and only if the pairs in $X$ correspond to placing $t+1$ in $\pi$, that is $X$ is empty or $X=\{(\pi(i),t+1),(\pi(i+1),t+1),\ldots,(\pi(t),t+1)\}$ for some $1\le i\le t$.
Otherwise there will be some $i$ such that $(\pi(i),t+1)\in X, (\pi(i+1),t+1)\notin X$. Let us pick the largest such $i$ to define the matching.

If $\pi(i)> \pi(i+1)$ is an inversion, that is $(\pi(i),\pi(i+1))\in T=Inv(\pi)$, the set $T'=T\setminus (\pi(i),\pi(i+1)) $ is the set of inversions of a permutation $\pi'$ where $\pi(i)$ and $\pi(i+1)$ has swapped places. Also define $X'=X\backslash \{(\pi(i),t+1)\}\cup \{(\pi(i+1),t+1)\}$. We now match $T\cup X$ with $T' \cup X'$.  Note that we remove $ (\pi(i),\pi(i+1)),(\pi(i),t+1)$ and add $(\pi(i+1),t+1)$ which is a valid matching since $\pi(i)> \pi(i+1)$. 

For example, take $\pi=3412$ in one-line notation, $t=4$ and $X=\{(54),(52)\}$. In this case $i=2,$ so $T'=T \setminus \{(41)\}$ and $X'=\{(52),(51)\}.$

If $\pi(i)< \pi(i+1)$ is not an inversion we instead form $T'$ by adding $(\pi(i),\pi(i+1)) $ and $X'=X\backslash (\pi(i),t+1)\cup (\pi(i+1),t+1)$. Again we match $T\cup X$ with $T' \cup X'$.  

The definition of $i$ as the largest position for which $(\pi(i),t+1)\in X, (\pi(i+1),t+1)\notin X$ does not change and thus these two cases form each others inverses and this is a valid matching for all sets $S$ that are not the set of inversions for some permutation. This completes the inductive step.
\end{proof}

  Next, we consider a special case of $\mathbf{R}_t^{\triangle}$ above, where $t=k-1$. For any $U\subseteq [1,k]$, we let $i_1:=\min U$, $j_1:=\min\{j\notin U:i_1<j\}$.
 Then recursively $i_r:=\min\{i\in U:j_{r-1}<i\}$, $j_r:=\min\{j\notin U:i_r<j\}$.
 We think of $i_r,j_r$ as an interval $[i_r,j_r]$ which corresponds to the product
 $R_{i_r,i_r+1}\cdots R_{i_r,j_r}$ and $U$ consisting of those intervals.
 Then we set $S(U)=\cup_{r=1}^m\{(i_r,i_r+1)\ldots,(i_{r},j_{r})\}$ to be the 
inversions within each interval that we are interested in. 

So each such $S(U)$ corresponds to a product of intervals \[R_{S(U)}:=\prod_{s=1}^m R_{i_s,i_s+1}\cdots R_{i_s,j_s}.\]

 We want to show that the specialization of Theorem \ref{thm:inversions} for $t=k-1$ gives the following sum over $2^k$ terms.
 \begin{lemma} \label{trieqk}
 \begin{equation} \label{algcancel}
 \mathbf{R}^\triangle_{k-1}=\sum_{U\subseteq [k]} (-1)^{|U|} R_{S(U)}
 \end{equation}
 
 \end{lemma}
 
\begin{proof}
When $t=k-1$, we are applying the operators to $k-1$ parts of length $1$ in the partition. The specialization of Theorem \ref{thm:inversions} to this situation demands that we must only sum over permutations $\pi$ for which $Inv(\pi)$ has at most one inversion $(i,k-1)$, since there is only one box to move from position $k-1$. Now inductively this implies that there can be at most one inversion $(i,j)$ for any $j$. If $(i_1,j),(i_2,j)\in Inv(\pi)$, then we would also need $(j,\ell)\in Inv(\pi)$, because if two boxes are moved from $j$ one box has to come from other position. But this would require $\ell,j,i_2,i_1$ in that order in one line notation of $\pi$ contradicting the inductive assumption of at most 1 inversion with $\ell$.
\end{proof}
 
 Eqn. \eqref{algcancel} can be written as
 \[1- \sum_{i=1}^{n-1} R_{i,i+1} + \sum_{i=1}^{n-2} R_{i,i+1}R_{i,i+2} + \sum_{i=1}^{n-2} \sum_{j > i+1} R_{i,i+1}R_{j,j+1}-\cdots. \]

For example, if $k=2,$ we have
 \[1-R_{12}.\]
When $k=3,$
\[1-R_{12}-R_{23}+R_{12}R_{13}\]
When $k=4,$
\[1-R_{12}-R_{23}-R_{34}+R_{12}R_{13}+R_{23}R_{24}+R_{12}R_{34}-R_{12}R_{13}R_{14}. \]
When $k=5,$

\begin{equation*}
\begin{split}
& 1-R_{12}-R_{23}-R_{34}-R_{45}+R_{12}R_{13}+R_{23}R_{24}+R_{34}R_{35}+R_{12}R_{34}+R_{12}R_{45}+R_{23}R_{45} \\
&-R_{12}R_{13}R_{14}-R_{12}R_{13}R_{45}-R_{12}R_{34}R_{35}-R_{23}R_{24}R_{25}+R_{12}R_{13}R_{14}R_{15}
\end{split}
\end{equation*}

Now we apply the triangle operators to a general integer vector. We will show that the only non-trivial case is that all entries of the vector are the same is when the entries are confined to a triangle (Lemma \ref{alleq}). Let us first define
\begin{definition} \label{trunc}
    Let $\mu=(\mu_1,\ldots,\mu_r) \in  \mathbb{Z}^{r}_+.$ We call
\[\hat{\mu}=\mu-(\tn{min} \, \mu) \cdot (1,\ldots,1).\]
a truncation of a non-negative vector $\mu$. 
\end{definition}
Motivation to define such an object is that we view $\mu$ as a vector obtained after applying a set of operators to a segment of a $k$-bounded partition. Recall the definition of $R_X$ in \eqref{r_x} for any $X \subset \mathbb{Z}^2_{+},$ which will be used below.

\begin{lemma} \label{triangle-lemma}
Let $\mu \in \mathbb{Z}_+^{t+1}.$ If $\hat{\mu}=(\hat{\mu}_1,\ldots,\hat{\mu}_{t+1})=(0,\ldots,0,m)$ such that $1 \leq m \leq t,$ we have $\mathbf{R}^\triangle_{t}(\mu)=0.$ 
\end{lemma}

\begin{proof} The proof is by induction over $t$. For all $t,$ we distinguish the two cases: (1) $1 \leq m \leq t-1$ and (2) $m=t$.

For $t=1,$ the only possibility is $\hat{\mu}=(0,1),$ and $\mathbf{R}^\triangle_1=\{R_{12}\}.$ Therefore we match $(1,0)$ with $(0,1)$ where we apply $R_{12}$ in the former and do not apply any non-trivial operator in the latter. Since the parities differ by one, they cancel out each other. 

Suppose $t\geq 2,$  and consider the three regions for the operators:\\
The top-left triangle is 
\[T:= \left\{ (i,j) \in \mathbb{Z}^2_{+} \, : \,  t-m+1 \leq i \leq t , i+1 \leq j \leq t+1 \right\}.\]
Then we have the bottom triangle
\[D:= \left\{ (i,j) \in \mathbb{Z}^2_{+} \, : \,  1 \leq i \leq t-m-1 , i+1 \leq j \leq t-m \right\}.\]
Finally, we have the strip 
\[S:= \left\{ (i,j) \in \mathbb{Z}^2_{+} \, : \,  1 \leq i \leq t-m , t-m+1 \leq j \leq t+1 \right\},\]
see Figure \ref{triangle}. Note that the regions above define the labels of the associated operators, not the actual regions in the figure, and that the sum of the combinations of operators in $T$ is just $\mathbf{R}_{m}^{\triangle}$ shifted by $t-m$. 

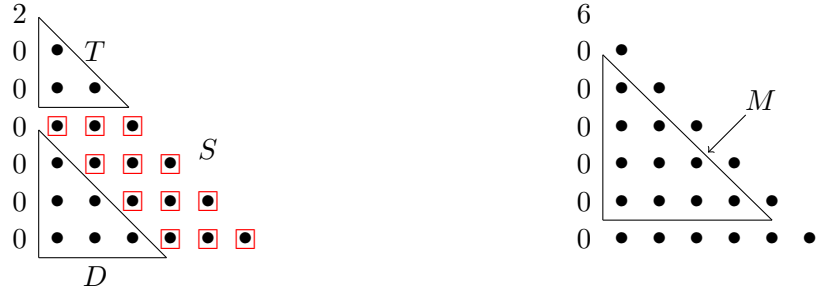
\begin{figure} 
\begin{tikzpicture}
\hspace*{1cm}

  \foreach \y in {1,1.5,...,3} {
        \draw (0, \y -1) node{$0$} ++(1,1);
    \foreach \x in {0.5,1,...,\y} {
      \draw (\x, 3-\y) node{$\bullet$} ++(1,1);
          }
  }
  
  \foreach \x in {0,0.5,1} {
    \foreach \y in {0,0.5,1,1.5} {
      \draw (2+\x-\y, \y) node{$\color{red}\square$} ++(1,1);
          }
  }

       \draw (0.5, 2.5) node{$\bullet$};
       \draw (0, 2.5) node{$0$}; 
       \draw (0, 3) node{$2$} ;
              \draw (1, 2.5) node{$T$} ;
              \draw (1, -0.5) node{$D$} ;
              \draw (2.5, 1.2) node{$S$} ;

      \draw (0.25,1.75) -- (1.45,1.75);
\draw (0.25,1.75) -- (0.25,2.95);
\draw (0.25,2.95) -- (1.45,1.75);

\draw (0.25,-0.25) -- (0.25,1.45);
\draw (0.25,-0.25) -- (1.95,-0.25);
\draw (0.25,1.45) -- (1.95,-0.25);

 \foreach \y in {1,1.5,...,3} {
        \draw (7.5, \y -1) node{$0$} ++(1,1);
    \foreach \x in {0.5,1,...,\y} {
      \draw (7.5+\x, 3-\y) node{$\bullet$} ++(1,1);
          }
  }

 \draw (8, 2.5) node{$\bullet$};
       \draw (7.5, 2.5) node{$0$}; 
       \draw (7.5, 3) node{$6$} ;
     
 \draw (7.75,0.25) -- (10,0.25);
\draw (7.75,0.25) -- (7.75,2.45);
\draw (7.75,2.45) -- (10,0.25);

              \draw (9.85, 1.85) node{$M$} ;
\draw [->] (9.65,1.65) -- (9.15,1.15);

\end{tikzpicture}
\caption{An example for Lemma \ref{triangle-lemma}. $t=6$ and $m=2$ for the figure on the left while $t=m=6$ for the figure on the right.}
\label{triangle}
\end{figure}

Starting with the first case $m \leq t-1,$ the induction hypothesis implies that there exists a matching within $T$, that is to say, for all $X \subseteq T$ there exists a unique $X' \subseteq T$ such that
\begin{equation} \label{match}
 R_X(\mu) \leftrightarrow R_{X'}(\mu) = \pi(R_X(\mu))   
\end{equation}
for some $\pi \in \mathcal{S}_{m+1}.$ Note that we can choose $X'=X$ if $R_X$ is not applicable to $\mu.$ For example, if $\mu_j=1$, $X=\{(i_1,j),(i_2,j)\}$ and $(j,s)\notin X$ for any $s$,  then $R_X$ is not applicable to $\mu$ because there are no two boxes to displace at the $j$th part of $\mu.$ See 
Theorem \ref{thm:inversions} where $2^{\binom{n}{2}}-n!$ possible choices for subsets of operators that are not applicable.

Now we consider $Y \subseteq S.$ We first partition $Y$ with respect to the first coordinates of its elements; let $Y_l =Y \cap (\{l\}\times \mathbb{Z}_+).$ Then we define $\pi(Y_l)$, for the same $\pi\in S_{m+1}$ as above, by letting 
\[\pi(Y_l)=\{(l,j) \in S \, : \, (l,\pi^{-1}(j)) \in Y_l\}.\] 
 Observe that $R_{Y_l}$ and $R_{\pi(Y_l)}$ have the same parity, and we have the matching 
    \[R_{Y_1} \circ \cdots \circ R_{Y_{m-t}} \circ R_X(\mu)  +R_{\pi(Y_1)} \circ \cdots \circ R_{\pi(Y_{m-t})}  \circ \pi(R_X(\mu))=0.\]
 Then for any $Z \subseteq D,$ we can apply $R_Z$ to both sides of the matching above and it will still be a matching. Thus we are done with the first case.

Now we look at the case $m=t.$ We consider the operators in the middle of the triangle $T,$
\[M= \left\{(i,j) \in \mathbb{Z}^2_{+} \, : \,  2 \leq i \leq t-1 , i+1 \leq j \leq t \right\},\]
see the right image in Figure \ref{triangle}.
Then we look at the operators on the sides
\begin{align*}
K:&= \left\{ (i,t+1) \in \mathbb{Z}^2_{+}\, : \, 2 \leq i \leq t  \right\} \cup \left \{ (1,j) \in \mathbb{Z}^2_{+} \, : \, 2 \leq j \leq t-1\right\} \cup \{(1,t)\} \\
&= \left\{(2,t+1),(3,t+1),\ldots, (t,t+1),(1,2),(1,3),\ldots, (1,t-1), (1,t) \right \}.
\end{align*}

For any $X \subseteq M$ and $Y \subseteq K,$ we will find some $Y' \in K$ such that
\[R_Y \circ R_X (\hat{\mu}) + R_{Y'} \circ R_X (\hat{\mu}) = 0.\]

Note that $X$ will not alter the top part nor the bottom part. The bottom part of 
$R_Y \circ R_X (\hat{\mu})$ will be equal to the number of $R_{1j}\in Y$ and the size of the top part will be $t$ minus the number of $R_{it}\in Y$. Boxes can be moved from the top part to the bottom part by  pairs 
$R_{1,i}R_{i,t+1}, 2\le i\le t-1$ and by $R_{1,t+1}$. We construct $Y'$ by interchanging all pairs $R_{1,i}R_{i,t+1}$ and also $R_{1,t+1}$ to change the parity. This gives the same partition with top and bottom parts interchanged and reverse parity.
In symbols $Y'$ is constructed as follows. For $i=2,\ldots, t,$
\begin{align*}
\tn{ If } &\left\{(i,t+1), (1, i)\right\} \subseteq Y, \tn{ then }  \left\{(i,t+1), (1, i)\right\} \cap Y'=\emptyset,\\
\tn{ while if } &\left\{(i,t+1), (1, i)\right\} \cap Y = \emptyset, \tn{ then }  \left\{(i,t+1), (1, i)\right\} \subseteq Y',\\
\tn{ otherwise } &\left\{(i,t+1), (1, i)\right\} \cap Y = \left\{(i,t+1), (1, i)\right\} \cap Y'.
\end{align*}
Finally, we take 
\[(1,t) \in Y \tn{ if and only if }(1,t) \notin Y'.\]
This last step guaranties $R_Y$ and $R_{Y'}$ have different parities. This completes the inductive step and the proof.
\end{proof}

\begin{lemma} \label{alleq}
Let $\mu \in \mathbb{Z}_+^{t+1}.$ If $  \hat{\mu}$ is non-zero and $\hat{\mu}_i \leq i-1$ for all $i$, for every additive term $R_Z$ in the expansion of $\mathbf{R}_{t}^{\triangle}(\mu)$ in \eqref{r_triangle}, there exists another term $R_{Z'}$ such that $R_{Z}(\mu)+R_{Z'}(\mu)=0.$ In particular, $\mathbf{R}_{t}^{\triangle}(\mu)=0.$ 
\end{lemma}
\begin{proof}
    For $1 \leq m \leq t,$ let $\mu_{m+1}$ be the part with the smallest index such that $\hat{\mu}_{m+1} \neq 0.$ We consider the following regions of operators:
\[T= \left\{ (i,j) \in \mathbb{Z}^2_{+} \, : \,  1 \leq i \leq m , i+1 \leq j \leq m+1 \right\}, \]
then the middle strip
\[M= \left\{ (i,j) \in \mathbb{Z}^2_{+} \, : \,  1 \leq i \leq m+1 , m+2 \leq j \leq t+1 \right\}. \]
We further divide $M$ into diagonals as
\[M_l=\left\{ (i,l+m+1) \in \mathbb{Z}^2_{+} \, : \,  1 \leq i \leq m+1 \right\}\]
for $l=1,\ldots, t-m.$  

\begin{figure}[h!]
\begin{tikzpicture}
%\hspace*{4cm}

  \foreach \x in {0,0.5,...,4} {
    \foreach \y in {0,0.5,1,1.5} {
      \draw (\x-\y, \y) node{$\bullet$} ++(1,1);
          }
  }

 \draw (-1.3, 0.5) node{$T$} ;
             %\draw (4.5, 1) node{$Q$} ;
              \draw (0, -0.4) node{$M_1$} ;
              \draw (0.6, -0.4) node{$M_2$} ;
              \draw (2.1, -0.4) node{$\ldots$} ;
              \draw (4.2, -0.4) node{$M_9$} ;

      \draw (-1.75,-0.15) -- (-0.3,-0.15);
\draw (-1.75,-0.15) -- (-1.75,1.45);
\draw (-1.75,1.45) -- (-0.3,-0.15);

    \foreach \y in {0,0.5,1} {
            \draw (-2, \y) node{$0$} ++(1,1);
  }

            \draw (-2, 1.5) node{$3$};

\draw (4.15,0.15) -- (2.65,1.65);
  
%M_l
   \foreach \x in {0,0.5,...,3.5} {
      \draw (\x+0.35,-0.15) -- (\x-1.35,1.65);      
  }

\end{tikzpicture}
\end{figure}

\noindent First consider $m=1.$ The only non-trivial case is $\hat{\mu}=(0,1).$ Therefore we match $(1,0)$ with $(0,1)$ as in the proof of the previous lemma. Then for any $X_l \in M_l$ for all $l,$ observing that $M_l$ has two elements, we can take $X_l'=(1 2) X_l$ where $(1 2)$ is the transposition that flips the entries in $X_l$. Then we have   

\[R_{X_l}(0,1)  + R_{X_l'}(1,0)=0.\]

\noindent 
For $m\ge 2$ we have a truncated vector of the form
  $ \hat{\mu}=(0,\ldots,0,n)$ where $ \hat{\mu}$ has $m+1$ parts and $n\leq m.$ By the proof of the lemma above, for all $X\subseteq T,$ there exists $X'=\pi(X)$ for some $\pi \in \mathcal{S}_m$ such that
\[R_X( \mu)+ R_{X'}( \mu)=0.\] 
Observe that we do not rule out the possibility that $X'=\pi(X).$ 
  
Then for any $Y_l \subseteq M_l,$ arguing similar to the proof above, we have
 
  \[R_{Y_1} \circ \cdots \circ R_{Y_{m-t}} \circ R_X( \mu)  +R_{\pi(Y_1)} \circ \cdots \circ R_{\pi(Y_{m-t})} \circ R_{X'}( \mu)=0\]
as before. Since any term in $\mathbf{R}_{t}^{\triangle}(\mu)$ is of the form $R_{Y_1} \circ \cdots \circ R_{Y_{m-t}} \circ R_X,$ this completes the proof.
\end{proof}

The following lemma is an application of Lemma \ref{alleq} on partitions with a column of $1$s. 
\begin{lemma} \label{long_col1}
Let $\lambda \in \mathcal{P}_k$ such that $l_1(\lambda)=k-1,$ i.e., $\lambda_{l(\lambda)}=\cdots=\lambda_{l(\lambda)-k+2}=1$ and $\lambda_{l(\lambda)-k+1} >1.$ Then for any $X \subset \{(i,j) \in \mathbb{Z}^2 \, : \, 1 \leq i \leq l(\lambda)-k+1,i+1 \leq j \leq k-\lambda_i+i\}$ \tn{(}see the range of the product \eqref{deter} for all parts except those of length $1$\tn{)} containing an $(i,j)$ for some $j\geq l(\lambda)-k+2,$ we have $\mathbf{R}_{k-1}^{\triangle} \circ R_X(\lambda)=0.$ 
\end{lemma}
\pf First, observe that for any $(i,j) \in X,$ we have $j < l(\lambda)$, so the box on the top row of $\lambda$ is not removed. Secondly, also by the assumption on $X,$ there must be some box moved from the parts of length $1$. Let the smallest index of the part that a box is removed is $l(\lambda)-w$ where $1\leq w \leq k-2.$ Then it follows that, by Def. \eqref{trunc}, we have $\hat{R_X}[(\lambda_{l(\lambda)-w},\ldots,\lambda_{l(\lambda)})]=R_X[(\lambda_{l(\lambda)-w},\ldots,\lambda_{l(\lambda)})] \in \mathbb{Z}_{+}^{w+1}$ is non-zero. 

Now, if $w=k-2,$ so the removed box is at the bottom of the column of ones, by Lemma \ref{alleq}, $\mathbf{R}_{k-1}^{\triangle} \circ R_X(\lambda)=0.$ However, if $w<k-2,$ we need to consider both $\mathbf{R}_{w}^{\triangle}$ and the operators associated with parts $\lambda_{l(\lambda)-k+2},\ldots,\lambda_{l(\lambda)-w}.$ In this case, we repeat a part of the proof of Lemma \ref{triangle-lemma}. More specifically, in the inductive step of the proof, we take $m=w$ for $t=k-1,$ the set of operators in $T$ corresponds to $\mathbf{R}_{w}^{\triangle}$, see Figure \ref{triangle}. 
By Lemma \ref{alleq}, there exists a matching within $T$ as in \eqref{match}. Then the rest of the argument in the proof for the operators in $S$ and $D$ applies in the same way, and we get the desired result.

\bbox

The same argument also applies to partitions with longer columns of $1.$ For our particular interest, we need the following.
\begin{lemma} \label{long_col2}
Let $\lambda \in \mathcal{P}_k$ such that $l_1(\lambda)=k.$ Then for any $X \subset \{(i,j) \in \mathbb{Z}^2 \, : \, 1 \leq i \leq l(\lambda)-k,i+1 \leq j \leq k-\lambda_i+i\}$ containing an $(i,j)$ for some $j\geq l(\lambda)-k+2,$ we have $\mathbf{R}_{k}^{\triangle} \circ R_X(\lambda)=0.$ 
\end{lemma}
\pf Observe that the top part is not moved by $X$ as before and some other part of length $1$ is moved. Therefore, the proof of Lemma \ref{long_col1} yields the same result.

\bbox

\bibliographystyle{alpha}
\bibliography{kcore}

\end{document}